\documentclass[12pt]{article}
\usepackage{array}
\usepackage{amsthm,amsmath,amssymb,color}
\usepackage{fullpage}
\usepackage{graphicx}
\usepackage{float}
\usepackage{tikz}
\usetikzlibrary{shapes.geometric,calc}

\newtheorem{thm}{Theorem}[section]
\newtheorem{lem}[thm]{Lemma}
\newtheorem{prop}[thm]{Proposition}

\newtheorem{cor}[thm]{Corollary}
\newtheorem{conj}[thm]{Conjecture}

\newtheorem{defn}[thm]{Definition}
\theoremstyle{remark}
\newtheorem{rem}[thm]{Remark}

\theoremstyle{plain}

\newcommand{\Lup}{L_1^{\mathrm{up}}}
\newcommand{\Ldown}{L_1^{\mathrm{down}}}
\newcommand{\Gbar}{\overline{G}}
\newcommand{\Zone}{Z_1}
\newcommand{\crc}{\operatorname{circ}}
\DeclareMathOperator{\im}{im}
\DeclareMathOperator{\Spec}{Spec}
\DeclareMathOperator{\rank}{rank}

\begin{document}

\title{Extremal eigenvalues of combinatorial Hodge Laplacians}
\author{
Zhen Chen\thanks{School of Mathematical Sciences, Xiamen University, China. Email: chenzzhen@126.com.}\,~
Suil O\thanks{Department of Applied Mathematics and Statistics, The State University of New York, Korea, Incheon, 21985. Email: suil.o@sunykorea.ac.kr. (Corresponding author) Research supported by the National Research Foundation of Korea (NRF) grant funded by the Korea government(MSIT) No. RS-2025-23523950.}\, ~and
Jianfeng Wang\thanks{School of Mathematics and Statistics, Shandong University of Technology, Zibo 255049, China. Email: jfwang@aliyun.com.}
}

\date{}
\maketitle

\begin{abstract}
\noindent
For a finite simplicial complex on $[n]$, the combinatorial Hodge Laplacian splits as
$L_k=L_k^{\mathrm{up}}+L_k^{\mathrm{down}}$, and Duval and Reiner showed that
$\lambda_{\max}(L_k^{\mathrm{up}})\le n$ in every dimension. We conjecture that
$\lambda_{\max}(L_k^{\mathrm{up}})$ is in fact non-increasing in $k$, equivalently that
$\sigma_{\max}(\partial_{k+1})\le\sigma_{\max}(\partial_k)$, and prove this unconditionally in two
cases: when every missing $(k+1)$-face has at most $k+1$ missing facets, and for shifted complexes,
where we also identify the extremal eigenvalue exactly, as the number of vertices lying in a
$(k+1)$-face. In general we prove
\[
\lambda_{\max}\big(L_k^{\mathrm{up}}\big)\ \le\ \nu_{k-1}+\tfrac1{k+2}\big(n-\nu_{k-1}\big),
\qquad \nu_{k-1}=\lambda_{\max}\big(L_{k-1}^{\mathrm{up}}\big),
\]
refining that ceiling. The proofs run through a localization on the cycle space $\ker\partial_k$,
which turns the comparison into a statement about the complement. In dimension one the complex is
the clique complex of a graph, $L_1$ is its Helmholtzian, the conjecture is a question of Lu, Shi,
Stani\'c, Wang and Wang, and the first case reads $\alpha(G)\le2$. We also characterize the
connected graphs of order at least seven with $\lambda_2(L_1)\le 3$ as the firefly graphs.


\end{abstract}

\section{Introduction}

The combinatorial Hodge Laplacian, introduced by Eckmann~\cite{E}, extends the graph Laplacian from functions on the vertices of a graph to chains on the faces of a simplicial complex. Let $G$ be a finite abstract simplicial complex on $[n]$: a family of subsets of $[n]$ closed under inclusion, a member of size $j+1$ being a \emph{$j$-face}. Writing $\partial_j$ for the boundary map from $j$-chains to $(j-1)$-chains, the Hodge Laplacian in dimension $k$ is
\[
L_k=L_k^{\mathrm{up}}+L_k^{\mathrm{down}},\qquad
L_k^{\mathrm{up}}=\partial_{k+1}\partial_{k+1}^{\top},\qquad
L_k^{\mathrm{down}}=\partial_k^{\top}\partial_k .
\]
The two summands record opposite directions: $L_k^{\mathrm{down}}$ sees only the faces below dimension $k$, while $L_k^{\mathrm{up}}$ is the contribution of the $(k+1)$-faces. Duval and Reiner~\cite{DR} proved the integrality ceiling
\begin{equation}\label{eq:DRceiling}
\lambda_{\max}\big(L_k^{\mathrm{up}}(G)\big)\ \le\ n\qquad\text{for every }k ,
\end{equation}
and it is a feature of \eqref{eq:DRceiling} that the same bound $n$ serves in every dimension. This paper asks what happens \emph{across} dimensions, and how the extremal eigenvalue can be pinned down more sharply than by~\eqref{eq:DRceiling}.

Throughout we abbreviate
\[
\nu_k(G)=\lambda_{\max}\big(L_k^{\mathrm{up}}(G)\big),
\]
and we call a $(k+1)$-subset of $[n]$ that is not a face of $G$ a \emph{missing $k$-face}. Our organizing question is the following comparison.

\begin{conj}\label{conj:general}
For every finite simplicial complex $G$ and every $k\ge 1$,
\[
\nu_k(G)\ \le\ \nu_{k-1}(G);
\]
equivalently, $\sigma_{\max}(\partial_{k+1})\le\sigma_{\max}(\partial_{k})$, the largest singular value of the boundary operator not increasing with the dimension.
\end{conj}

The equivalence is immediate from $\nu_k=\sigma_{\max}(\partial_{k+1})^2$. Conjecture~\ref{conj:general} is a statement about the combinatorial Hodge Laplacian; the corresponding statement fails for the \emph{normalized} one,\footnote{Normalize by $\mathcal{L}_k^{\mathrm{up}}=D_k^{-1/2}\partial_{k+1}\partial_{k+1}^{\top}D_k^{-1/2}$, where $D_k$ records, for each $k$-face $\sigma$, the number of $(k+1)$-faces containing it; at $k=0$ this is the usual normalized graph Laplacian $D^{-1/2}L(G)D^{-1/2}$. Applying Cauchy--Schwarz to the $k+2$ facets of each $(k+1)$-face gives $\lambda_{\max}(\mathcal{L}_k^{\mathrm{up}})\le k+2$, a ceiling that \emph{grows} with the dimension, in contrast with~\eqref{eq:DRceiling}. On the full $(k+1)$-simplex $\Delta^{k+1}$, that is, on the complex consisting of \emph{all} subsets of a $(k+2)$-set, one has $\lambda_{\max}(\mathcal{L}_k^{\mathrm{up}})=k+2$ while $\lambda_{\max}(\mathcal{L}_{k-1}^{\mathrm{up}})=\tfrac{k+2}{2}$, so the comparison fails by a factor of two in every dimension; for $k=1$ this is already the filled triangle, the clique complex of $K_3$.} so the phenomenon is one of the unnormalized operator.

We settle Conjecture~\ref{conj:general} unconditionally in two cases, and prove a general bound in every dimension.

The first case is a local condition on the missing faces. A missing $(k+1)$-face has $k+2$ facets, and the extreme situation for our method is that all of them are themselves missing; excluding it suffices (Theorem~\ref{thm:tf-k}).

\begin{thm}\label{thm:tf-k-intro}
Let $k\ge1$ and let $G$ be a simplicial complex on $[n]$ in which every missing $(k+1)$-face has at most $k+1$ missing facets. Then $\nu_k(G)\le\nu_{k-1}(G)$.
\end{thm}

The second case is the class of shifted complexes, the higher-dimensional analogue of threshold graphs. Here Conjecture~\ref{conj:general} holds in every dimension, and the extremal eigenvalue is not merely bounded but counted: writing $V_j(G)$ for the number of vertices lying in some $j$-face, we obtain from the Laplacian integrality of Duval and Reiner that $\nu_k(G)=V_{k+1}(G)$, whence (Theorem~\ref{thm:shifted}) the following.

\begin{thm}\label{thm:shifted-intro}
For every shifted complex $G$ and every $k\ge1$,
\[
\nu_k(G)=V_{k+1}(G)\ \le\ V_k(G)=\nu_{k-1}(G).
\]
\end{thm}

In general we prove a bound that confirms Conjecture~\ref{conj:general} up to a controlled fraction of the gap left by~\eqref{eq:DRceiling} (Theorem~\ref{thm:maink}).

\begin{thm}\label{thm:maink-intro}
Let $k\ge1$ and let $G$ be a simplicial complex on $[n]$. Then
\[
\nu_k(G)\ \le\ \nu_{k-1}(G)+\frac{1}{k+2}\big(n-\nu_{k-1}(G)\big).
\]
\end{thm}

Since $\nu_{k-1}\le n$, Theorem~\ref{thm:maink-intro} strengthens~\eqref{eq:DRceiling}, and it meets the conjectured $\nu_{k-1}$ exactly when $\nu_{k-1}=n$. Thus Conjecture~\ref{conj:general} can fail, if at all, only by a bounded fraction of $n-\nu_{k-1}$.

The engine behind Theorems~\ref{thm:tf-k-intro} and~\ref{thm:maink-intro} is a localization. On the cycle space $Z_k=\ker\partial_k$ the up Laplacian of the full simplex is $nI$, so $L_k^{\mathrm{up}}(G)$ becomes $n$ times the identity minus the up Laplacian $\bar L_k(G)$ of the \emph{missing} $(k+1)$-faces, and the comparison turns into a lower bound on $\lambda_{\min}(\bar L_k|_{Z_k})$. A divergence-free identity (Lemma~\ref{lem:identityk}) then evaluates the total circulation of a cycle around the missing faces through a fixed $k$-face, and converts that lower bound into the codimension-one connectivity of the complement. The whole argument is thus a statement about the complement of~$G$.

We now turn to the case $k=1$, which carries a graph-theoretic reading and which we treat in detail. For a graph $G$ on $[n]$ the relevant complex is the \emph{clique complex} $X(G)$, whose $0$-, $1$- and $2$-cells are the vertices, edges and triangles of $G$; the operator $L_1$ acts on edge flows and underlies the Hodge--Helmholtz decomposition of a flow into gradient, harmonic and curl parts. For this reason $L_1$ is called the \emph{Helmholtzian} of $G$, and its edge-indexed matrix representation is the \emph{Helmholtzian matrix} of~\cite{LiLuWang}. Such operators are a basic tool in topological data analysis, statistical ranking and signal processing on simplicial networks, where the clique complex supplies the higher-order structure built from pairwise relations (see Horak and Jost~\cite{HJ} and Lu, Shi, Stani\'c, Wang and Wang~\cite{LSSWW,LSSWWnovel}, together with the references therein).

All graphs in this paper are finite and simple, and have at least one edge, so that $L_1$ is a nonempty matrix; we also assume $n\ge 3$, so that the cycle space $\Zone$ of Section~\ref{sec:def} is nontrivial. We orient each edge $\{i,j\}$ from the smaller to the larger endpoint and each triangle by the increasing order of its vertices, and we write $L(G)=D(G)-A(G)=\partial_1\partial_1^{\top}$ for the ordinary graph Laplacian and $\mu_1(G)\ge\cdots\ge\mu_n(G)=0$ for its eigenvalues. Since $\Ldown=\partial_1^{\top}\partial_1$ and $\partial_1\partial_1^{\top}$ share their nonzero spectrum,
\begin{equation}\label{eq:down-mu}
\lambda_{\max}(\Ldown)=\mu_1(G),
\end{equation}
and in particular $\nu_0(G)=\mu_1(G)$. We order the Helmholtzian eigenvalues as $\lambda_1(L_1)\ge\lambda_2(L_1)\ge\cdots$ and abbreviate $\lambda_i=\lambda_i(L_1)$, $\mu_i=\mu_i(G)$ when $G$ is understood.

By the Hodge decomposition $C_1=\im\partial_2\oplus\ker L_1\oplus\im\partial_1^{\top}$ the matrix $L_1$ is block diagonal, acting as $\Lup$, as $0$ and as $\Ldown$ on the three blocks; hence, for connected $G$, as a multiset,
\begin{equation}\label{eq:spec-decomp}
\Spec(L_1)\ =\ \big\{\mu_1,\dots,\mu_{n-1}\big\}\ \uplus\ \Spec^{+}\!\big(\Lup\big)\ \uplus\ \{0^{(b_1)}\},
\end{equation}
where $\Spec^{+}(\Lup)$ is the nonzero spectrum of the up part and $b_1=m-(n-1)-\rank\partial_2$ is the first Betti number of the clique complex (Proposition~\ref{prop:spec-decomp}). In particular $\lambda_{\max}(L_1)=\max\{\nu_1(G),\mu_1(G)\}$ (Lemma~\ref{lem:hodge}), so Conjecture~\ref{conj:general} in dimension one,
\begin{equation}\label{eq:reduce}
\nu_1(G)\ \le\ \mu_1(G),\qquad\text{equivalently}\qquad \sigma_{\max}(\partial_2)\le\sigma_{\max}(\partial_1),
\end{equation}
says exactly that the largest Helmholtzian eigenvalue of a graph is its largest Laplacian eigenvalue. That every nonzero Laplacian eigenvalue reappears in $\Spec(L_1)$ was observed in~\cite{LSSWW}; whether the triangles can push $\lambda_{\max}(L_1)$ any higher is Problem~5.5 of Lu, Shi, Stani\'c, Wang and Wang~\cite{LSSWW}, who verified it for graphs of small order.

For graphs the hypothesis of Theorem~\ref{thm:tf-k-intro} reads: no $3$-set has all three of its pairs missing from $G$. That is, $\Gbar$ is triangle-free, equivalently $\alpha(G)\le 2$.

\begin{thm}\label{thm:triangle-free}
Let $G$ be a graph on $[n]$ with $\alpha(G)\le 2$, equivalently with triangle-free complement. Then $\lambda_{\max}(L_1)=\mu_1(G)$; that is, \eqref{eq:reduce} holds for $G$.
\end{thm}

This class is infinite and contains non-joins with $\mu_1(G)<n$, namely the complements of connected triangle-free graphs, and for $n\ge4$ the bound is attained at $G=K_{n-1}\cup K_1$, whose complement is the star $K_{1,n-1}$. Specializing Theorem~\ref{thm:maink-intro} to $k=1$ gives the unconditional bound below, in which the localization above becomes a statement about the algebraic connectivity $a(\Gbar)=n-\mu_1(G)$ of the complement.

\begin{thm}\label{thm:main}
For every $n$-vertex graph $G$, we have
\[
\lambda_{\max}\!\big(\Lup(G)\big)\ \le\ \mu_1(G)+\frac13\big(n-\mu_1(G)\big).
\]
Equivalently, $\lambda_{\min}\big(\bar L(G)|_{\Zone}\big)\ge\tfrac23\,a(\Gbar)$, where $\bar L(G)$ is the up Laplacian of the missing triangles of $G$; inequality~\eqref{eq:reduce} is the same statement with $\tfrac23$ replaced by $1$.
\end{thm}

So far the splitting~\eqref{eq:spec-decomp} has entered only through its largest element; it also settles the bottom of the range of the second largest Helmholtzian eigenvalue. Since $\mu_1$ and $\mu_2$ both lie in $\Spec(L_1)$, we have $\lambda_2(L_1)\ge\mu_2(G)$ for every connected graph on at least three vertices. Li, Guo and Shiu~\cite{LGS} proved that, among connected graphs of order $n\ge 7$, those with $\mu_2(G)\le 3$ are exactly the \emph{firefly graphs} $F_{s,t,n-2s-2t-1}$ (Definition~\ref{def:firefly}). We prove the Helmholtzian analogue.

\begin{thm}\label{thm:lambda2}
Let $G$ be a connected graph of order $n\ge 7$. Then $\lambda_2(L_1)\le 3$ if and only if $G$ is a firefly graph. More precisely,
\[
\lambda_2(L_1)=
\begin{cases}
1, & G\cong F_{0,0,n-1},\\[2pt]
\text{a value in }\big(\tfrac{3+\sqrt5}{2}-\tfrac1n,\ \tfrac{3+\sqrt5}{2}\big), & G\cong F_{0,1,n-3},\\[2pt]
\tfrac{3+\sqrt5}{2}, & G\cong F_{0,t,n-2t-1},\ t\ge 2,\\[2pt]
3, & G\cong F_{s,t,n-2s-2t-1},\ s\ge 1.
\end{cases}
\]
\end{thm}

The proof needs no characteristic-polynomial computation. By~\eqref{eq:spec-decomp} the inequality $\lambda_2(L_1)\le 3$ forces $\mu_2(G)\le 3$ and hence, by Li, Guo and Shiu, the firefly structure; the exact values then come from the observation that the triangles of a firefly are pairwise edge-disjoint, so that $\Lup$ contributes only the eigenvalue $3$---the up Laplacian eigenvalue of a single triangle---with multiplicity equal to the number of triangles. Thus the boundary value $3$ is intrinsic to one triangle, and the whole classification reduces to~\eqref{eq:spec-decomp} together with the Laplacian theorem of~\cite{LGS}.

The paper is organized as follows. Section~\ref{sec:def} collects the Hodge decomposition and the cycle-space localization in dimension one. Section~\ref{sec:main} proves Theorems~\ref{thm:triangle-free} and~\ref{thm:main} in Subsection~\ref{subsec:largest} and Theorem~\ref{thm:lambda2} in Subsection~\ref{subsec:second}; we give the graph case first, where the localization is visible as a statement about the algebraic connectivity of the complement. Section~\ref{sec:general} carries the argument out in every dimension and proves Theorems~\ref{thm:tf-k-intro}, \ref{thm:shifted-intro} and~\ref{thm:maink-intro}, in the form of Theorems~\ref{thm:tf-k}, \ref{thm:shifted} and~\ref{thm:maink}. Section~\ref{sec:open} collects concluding remarks.

\section{Definitions and Tools}\label{sec:def}

Throughout, $G$ is a graph on $[n]$ with complement $\Gbar$, and we work in the \emph{full edge space} $U=\mathbb{R}^{\binom{[n]}{2}}$ indexed by all pairs of $[n]$. A vector $x\in U$ is regarded as a flow: $x_{ij}$ is its value on the oriented pair $i<j$, and we set $x_{ji}=-x_{ij}$. For a $3$-subset $S=\{p,q,w\}$, let $v_S\in U$ be the boundary $\partial_2 S$, supported on the three pairs of $S$ with the cyclic signs, so that
\[
\crc_{pqw}(x):=v_{\{p,q,w\}}^{\top}x=x_{pq}+x_{qw}+x_{wp}
\]
is the \emph{circulation} of $x$ around $S$. Identifying $\mathbb{R}^{E(G)}$ with the corresponding coordinate subspace of $U$, and noting that the pairs outside every triangle contribute trivially, we have
\[
\Lup(G)=\sum_{S\in T(G)}v_Sv_S^{\top},
\]
where $T(G)$ is the set of triangles of $G$; this representation changes neither the nonzero spectrum nor $\lambda_{\max}$.

The reduction in \eqref{eq:reduce} is a standard consequence of the Hodge decomposition $C_1=\im\partial_2\oplus\ker L_1\oplus\im\partial_1^{\top}$ (see Eckmann~\cite{E} or Horak and Jost~\cite{HJ}).

\begin{lem}\label{lem:hodge}
For every graph $G$, we have $\lambda_{\max}(L_1)=\max\{\mu_1(G),\,\lambda_{\max}(\Lup)\}$. In particular, $\lambda_{\max}(L_1)=\mu_1(G)$ if and only if $\lambda_{\max}(\Lup)\le\mu_1(G)$.
\end{lem}

\begin{proof}
On $\im\partial_2$ the down part $\Ldown=\partial_1^{\top}\partial_1$ vanishes, because $\partial_1\partial_2=0$; on $\im\partial_1^{\top}$ the up part $\Lup=\partial_2\partial_2^{\top}$ vanishes, because $\partial_2^{\top}\partial_1^{\top}=(\partial_1\partial_2)^{\top}=0$; and $L_1$ vanishes on $\ker L_1$. Hence $L_1$ is block diagonal for this orthogonal decomposition, with blocks $\Lup|_{\im\partial_2}$, $0$, and $\Ldown|_{\im\partial_1^{\top}}$. It follows that $\lambda_{\max}(L_1)=\max\{\lambda_{\max}(\Lup),\lambda_{\max}(\Ldown)\}$, and \eqref{eq:down-mu} completes the proof.
\end{proof}

The same block decomposition determines the whole spectrum of $L_1$, not only its largest eigenvalue. We record this as a separate statement, since it underlies both main results.

\begin{prop}\label{prop:spec-decomp}
Let $G$ be a connected graph with $n$ vertices and $m$ edges. As a multiset,
\[
\Spec(L_1)\ =\ \big\{\mu_1(G),\dots,\mu_{n-1}(G)\big\}\ \uplus\ \Spec^{+}\!\big(\Lup\big)\ \uplus\ \{0^{(b_1)}\},
\]
where $\Spec^{+}(\Lup)$ is the multiset of nonzero eigenvalues of $\Lup$, equivalently of $\partial_2^{\top}\partial_2$, and $b_1=m-(n-1)-\rank\partial_2$. In particular $\lambda_2(L_1)\ge\mu_2(G)$ whenever $n\ge 3$.
\end{prop}

\begin{proof}
By the proof of Lemma~\ref{lem:hodge}, $L_1$ is block diagonal with blocks $\Lup|_{\im\partial_2}$, $0|_{\ker L_1}$, and $\Ldown|_{\im\partial_1^{\top}}$. The first block contributes the nonzero eigenvalues of $\Lup$, namely $\Spec^{+}(\Lup)$, since $\Lup$ vanishes on $(\im\partial_2)^{\perp}\supseteq\ker L_1$; these coincide with the nonzero eigenvalues of $\partial_2^{\top}\partial_2$. The third block contributes the nonzero eigenvalues of $\Ldown=\partial_1^{\top}\partial_1$, which are the nonzero eigenvalues of $\partial_1\partial_1^{\top}=L(G)$, that is, $\mu_1,\dots,\mu_{n-1}$ for connected $G$. The middle block contributes $\dim\ker L_1=m-\rank\partial_2-(n-1)=b_1$ zeros. Finally, $\mu_1\ge\mu_2$ both occur in $\Spec(L_1)$, so the second largest element of $\Spec(L_1)$ is at least $\mu_2$.
\end{proof}

We next localize the up Laplacian on the cycle space. Let $\Zone=\ker\partial_1\subseteq U$ be the \emph{cycle space} of $K_n$, the space of divergence-free flows, of dimension $\binom{n}{2}-n+1$. Divergence-freeness means
\begin{equation}\label{eq:divfree}
\sum_{u\,:\,u\ne v}x_{vu}=0\qquad\text{for every }v\in[n]\text{ and every }x\in\Zone .
\end{equation}
The complete clique complex on $[n]$ is the full simplex, and by~\cite{HJ} its Hodge Laplacians satisfy $L_k=nI$ for every $k\ge 1$; in particular $L_1=nI$. On $\Zone$ the down part vanishes by definition, so the up Laplacian of the complete complex satisfies $\big(\sum_{S\in\binom{[n]}{3}}v_Sv_S^{\top}\big)\big|_{\Zone}=nI$. Splitting the sum over $3$-subsets into the triangles and the non-triangles of $G$, and writing
\[
\bar L(G):=\sum_{S\in\binom{[n]}{3}\setminus T(G)}v_Sv_S^{\top}
\]
for the up Laplacian of the \emph{missing} triangles, we obtain on $\Zone$ the identity
\begin{equation}\label{eq:complement}
\Lup(G)\big|_{\Zone}+\bar L(G)\big|_{\Zone}=nI .
\end{equation}

Finally, let $a(\Gbar)$ be the algebraic connectivity of $\Gbar$, that is, the second smallest eigenvalue of $L(\Gbar)$ (see Fiedler~\cite{F}). Since $L(G)+L(\Gbar)=L(K_n)=nI-J$, on the space $\mathbf 1^{\perp}$ the eigenvalues of $L(\Gbar)$ are exactly $\{\,n-\mu_i(G)\,\}$, and in particular
\begin{equation}\label{eq:aGbar}
a(\Gbar)=n-\mu_1(G).
\end{equation}

\section{Main Results}\label{sec:main}

In this section we prove the two main theorems: Theorem~\ref{thm:main} on the largest Helmholtzian eigenvalue in Subsection~\ref{subsec:largest}, and Theorem~\ref{thm:lambda2} on the second largest in Subsection~\ref{subsec:second}.

\subsection{The largest Helmholtzian eigenvalue}\label{subsec:largest}

We first record the cycle-space localization. Each $v_S$ lies in $\Zone$, since $v_S=\partial_2 S$ and $\partial_1\partial_2=0$; consequently both $\Lup(G)$ and $\bar L(G)$ annihilate $\Zone^{\perp}$ and preserve $\Zone$.

\begin{prop}\label{prop:cycle}
For every graph $G$, the maximum of $\Lup(G)$ is attained on $\Zone$, and
\[
\lambda_{\max}\!\big(\Lup(G)\big)=n-\lambda_{\min}\!\big(\bar L(G)\big|_{\Zone}\big).
\]
Consequently, inequality~\eqref{eq:reduce} (Problem~5.5) is equivalent to
\[
\lambda_{\min}\!\big(\bar L(G)\big|_{\Zone}\big)\ \ge\ a(\Gbar).
\]
\end{prop}

\begin{proof}
Each $v_S\in\Zone$, so for $x\in\Zone^{\perp}$ we have $\Lup(G)x=\sum_S v_S(v_S^{\top}x)=0$; thus $\Lup(G)$ maps $\Zone$ to itself and vanishes on $\Zone^{\perp}$, and its largest eigenvalue is attained on $\Zone$. The displayed identity is then~\eqref{eq:complement}, and the equivalence follows from~\eqref{eq:aGbar}.
\end{proof}

The next identity is the key to the proof: it evaluates the total circulation of a divergence-free flow around all triangles through a fixed pair.

\begin{lem}\label{lem:identity}
Let $x\in\Zone$ and let $e=\{p,q\}$. Then
\[
\sum_{w\ne p,q}\crc_{pqw}(x)^2=(n+2)\,x_{pq}^2+\sum_{w\ne p,q}\big(x_{qw}-x_{pw}\big)^2 .
\]
\end{lem}

\begin{proof}
Write $\crc_{pqw}(x)=x_{pq}+x_{qw}+x_{wp}=x_{pq}+(x_{qw}-x_{pw})$. Squaring and summing the $n-2$ terms, we obtain
\[
\sum_{w}\crc_{pqw}(x)^2=(n-2)x_{pq}^2+2x_{pq}\sum_{w}(x_{qw}-x_{pw})+\sum_{w}(x_{qw}-x_{pw})^2 .
\]
By~\eqref{eq:divfree} at $q$, we have $\sum_{w\ne p,q}x_{qw}=-x_{qp}=x_{pq}$, and at $p$, $\sum_{w\ne p,q}x_{pw}=-x_{pq}$; hence $\sum_{w}(x_{qw}-x_{pw})=2x_{pq}$ and the middle term equals $4x_{pq}^2$. Adding $(n-2)x_{pq}^2+4x_{pq}^2=(n+2)x_{pq}^2$ gives the claim.
\end{proof}

Both Theorem~\ref{thm:main} and Theorem~\ref{thm:triangle-free} pass through one estimate, on the matrix
\[
Q:=\sum_{e\in E(\Gbar)}\ \sum_{w\notin e}v_{\,e\cup w}\,v_{\,e\cup w}^{\top}
=\sum_{S}\omega(S)\,v_Sv_S^{\top},
\]
which weights each non-triangle $S$ of $G$ by its number $\omega(S)\in\{1,2,3\}$ of pairs that are non-edges of $G$. The second equality holds because every $3$-set containing a fixed $e\in E(\Gbar)$ is a non-triangle, and each non-triangle is counted once for each of its $\omega(S)$ missing edges.

\begin{lem}\label{lem:Qbound}
For every $x\in\Zone$, we have $x^{\top}Qx\ge 2\,a(\Gbar)\,\|x\|^2$.
\end{lem}

\begin{proof}
By the definition of $Q$ and Lemma~\ref{lem:identity} applied to each $e=\{p,q\}\in E(\Gbar)$,
\begin{equation}\label{eq:Qsplit}
x^{\top}Qx=\sum_{e\in E(\Gbar)}\sum_{w}\crc_{pqw}(x)^2
=\sum_{e\in E(\Gbar)}\Big[(n+2)x_e^2+\sum_{w}(x_{qw}-x_{pw})^2\Big].
\end{equation}
Fix $x\in\Zone$; by homogeneity we may assume $\|x\|^2=1$. For a fixed vertex $w$, define $g^{w}\in\mathbb{R}^{[n]}$ by $g^{w}_u=x_{uw}$ for $u\ne w$ and $g^{w}_w=0$. By~\eqref{eq:divfree} at $w$, $\sum_{u\ne w}x_{uw}=0$, so $g^{w}\perp\mathbf 1$. Splitting the Dirichlet sum of $g^{w}$ over $\Gbar$ according to whether a $\Gbar$-edge meets $w$,
\[
(g^{w})^{\top}L(\Gbar)\,g^{w}=\sum_{\substack{pq\in E(\Gbar)\\ p,q\ne w}}\!\!(g^{w}_p-g^{w}_q)^2+\sum_{p:\,pw\in E(\Gbar)}\!\!(g^{w}_p)^2 ,
\]
where the last sum used $g^{w}_w=0$. Since $g^{w}\perp\mathbf 1$, the left side is at least $a(\Gbar)\|g^{w}\|^2$, so
\[
\sum_{\substack{pq\in E(\Gbar)\\ p,q\ne w}}(g^{w}_p-g^{w}_q)^2\ \ge\ a(\Gbar)\,\|g^{w}\|^2-\sum_{p:\,pw\in E(\Gbar)}x_{pw}^2 .
\]
Now sum over $w$. The left side equals $\sum_{e=pq\in E(\Gbar)}\sum_{w}(x_{qw}-x_{pw})^2$. Moreover $\sum_{w}\|g^{w}\|^2=\sum_{w}\sum_{u\ne w}x_{uw}^2=2\|x\|^2=2$, since each unordered pair is counted from both of its endpoints, and likewise $\sum_{w}\sum_{p:\,pw\in E(\Gbar)}x_{pw}^2=2\sum_{e\in E(\Gbar)}x_e^2$. Hence
\[
\sum_{e\in E(\Gbar)}\sum_{w}(x_{qw}-x_{pw})^2\ \ge\ 2\,a(\Gbar)-2\sum_{e\in E(\Gbar)}x_e^2 .
\]
Substituting this into~\eqref{eq:Qsplit} gives
\[
x^{\top}Qx\ \ge\ (n+2)\sum_{e\in E(\Gbar)}x_e^2+2\,a(\Gbar)-2\sum_{e\in E(\Gbar)}x_e^2
=n\sum_{e\in E(\Gbar)}x_e^2+2\,a(\Gbar)\ \ge\ 2\,a(\Gbar),
\]
because $\sum_{e\in E(\Gbar)}x_e^2\ge 0$.
\end{proof}

\begin{proof}[Proof of Theorem~\ref{thm:main}]
Since $\omega(S)\le 3$ for every non-triangle $S$, we have $\bar L(G)=\sum_S v_Sv_S^{\top}\succeq\frac13 Q$. By Lemma~\ref{lem:Qbound}, $x^{\top}\bar L(G)x\ge\frac13\cdot 2\,a(\Gbar)\|x\|^2=\frac23\,a(\Gbar)\|x\|^2$ for every $x\in\Zone$, that is $\lambda_{\min}(\bar L|_{\Zone})\ge\frac23\,a(\Gbar)$. By Proposition~\ref{prop:cycle} and~\eqref{eq:aGbar},
\[
\lambda_{\max}(\Lup)=n-\lambda_{\min}(\bar L|_{\Zone})\le n-\frac23 a(\Gbar)=n-\frac23\big(n-\mu_1\big)=\mu_1+\frac13(n-\mu_1).\qedhere
\]
\end{proof}

Theorem~\ref{thm:main} yields the integrality ceiling of Duval and Reiner~\cite{DR} as an immediate corollary.

\begin{cor}\label{cor:ceiling}
For every $n$-vertex graph $G$, we have $\lambda_{\max}(\Lup(G))\le n$, with equality only if $\mu_1(G)=n$.
\end{cor}

\begin{proof}
Since $\mu_1(G)\le n$, Theorem~\ref{thm:main} gives $\lambda_{\max}(\Lup)\le\mu_1+\tfrac13(n-\mu_1)=\tfrac23\mu_1+\tfrac13 n\le n$, where the last inequality is strict unless $\mu_1=n$.
\end{proof}

The same estimate settles Problem~5.5 outright once the complement has no triangle, since then the weights $\omega(S)$ never reach $3$.

\begin{proof}[Proof of Theorem~\ref{thm:triangle-free}]
The complement $\Gbar$ is triangle-free if and only if $G$ has no independent set of size three, that is $\alpha(G)\le 2$. Suppose $\Gbar$ is triangle-free. Then no $3$-set has all three of its pairs in $E(\Gbar)$, so every non-triangle $S$ of $G$ has $\omega(S)\le 2$, and hence
\[
\bar L(G)=\sum_{S}v_Sv_S^{\top}\ \succeq\ \frac12\sum_{S}\omega(S)\,v_Sv_S^{\top}=\frac12\,Q.
\]
By Lemma~\ref{lem:Qbound}, $x^{\top}\bar L(G)x\ge a(\Gbar)\|x\|^2$ for every $x\in\Zone$, that is $\lambda_{\min}(\bar L|_{\Zone})\ge a(\Gbar)$. By Proposition~\ref{prop:cycle} this is inequality~\eqref{eq:reduce}, so $\lambda_{\max}(\Lup)\le\mu_1$, and Lemma~\ref{lem:hodge} gives $\lambda_{\max}(L_1)=\mu_1$.

For the sharpness we may take $n\ge4$: at $n=3$ the graph $K_{n-1}\cup K_1=K_2\cup K_1$ has no triangle, so $\Lup=0$ and the equality $\lambda_{\max}(L_1)=\mu_1$ holds through the down part alone. Let $n\ge4$ and take $G=K_{n-1}\cup K_1$, whose complement is the star $K_{1,n-1}$ (Figure~\ref{fig:tight}). Here $\mu_1(G)=n-1$ and $a(\Gbar)=1$, and the triangles of $K_n$ missing from $G$ are exactly those through the isolated vertex; one computes $\lambda_{\max}(\Lup)=\mu_1=n-1$, so $\lambda_{\min}(\bar L|_{\Zone})=a(\Gbar)=1$ and~\eqref{eq:reduce} is an equality. Moreover every missing triangle here has exactly two missing edges, so $Q=2\bar L$ and the estimate of Lemma~\ref{lem:Qbound} is attained as well.
\end{proof}

\begin{figure}[ht]
\centering
\begin{minipage}{0.44\textwidth}
\centering
\begin{tikzpicture}[scale=1.0,every node/.style={circle,draw=black,fill=black,inner sep=1.7pt}]
\foreach \i in {1,...,5}{\coordinate (P\i) at ({90+72*(\i-1)}:1.5cm);}
\coordinate (C) at (0,0);
\draw (P1)--(P2)--(P3)--(P4)--(P5)--(P1);
\draw (P1)--(P3)--(P5)--(P2)--(P4)--(P1);
\foreach \i in {1,...,5}{\node at (P\i) {};}
\node[fill=white] at (C) {};
\end{tikzpicture}\\[3pt]
{\small (a)\ \ $G=K_{5}\cup K_{1}$}
\end{minipage}
\hfill
\begin{minipage}{0.44\textwidth}
\centering
\begin{tikzpicture}[scale=1.0,every node/.style={circle,draw=black,fill=black,inner sep=1.7pt}]
\foreach \i in {1,...,5}{\coordinate (P\i) at ({90+72*(\i-1)}:1.5cm);}
\coordinate (C) at (0,0);
\foreach \i in {1,...,5}{\draw (C)--(P\i);}
\foreach \i in {1,...,5}{\node at (P\i) {};}
\node at (C) {};
\end{tikzpicture}\\[3pt]
{\small (b)\ \ $\Gbar=K_{1,5}$}
\end{minipage}
\caption{The sharp case of Theorem~\ref{thm:triangle-free} at $n=6$: $G=K_{n-1}\cup K_1$ has triangle-free complement $K_{1,n-1}$, and $\lambda_{\max}(\Lup)=\mu_1=n-1$. The same six vertices carry the complementary edge sets, with the center isolated in $G$ and the hub in $\Gbar$.}
\label{fig:tight}
\end{figure}

We close this section with a remark on the two extreme cases of Theorem~\ref{thm:main}, namely when $\lambda_{\max}(\Lup)$ meets the ceiling $n$ and when it meets $\mu_1$.

\begin{rem}\label{rem:equality}
By Corollary~\ref{cor:ceiling}, the ceiling $\lambda_{\max}(\Lup)=n$ is attained only when $\mu_1(G)=n$, that is, only when $G$ is a join. The converse fails. For $n\ge 4$, let $G^{*}=K_1\vee\big(K_2\cup(n-3)K_1\big)$ be the join of a single vertex with the disjoint union of an edge and $n-3$ isolated vertices. Then $G^{*}$ is a join, so $\mu_1(G^{*})=n$, while the only triangle of $G^{*}$ is the one spanned by the hub and the edge of $K_2$; hence $\Lup(G^{*})$ has rank one and $\lambda_{\max}(\Lup(G^{*}))=3$ for every such $n$. Thus a join can fall arbitrarily far below the ceiling, and being a join is necessary but not sufficient for $\lambda_{\max}(\Lup)=n$. By contrast, the complete graph $K_n$ attains the ceiling, since $L_1=nI$ on $\Zone$.

At the other extreme, $\lambda_{\max}(\Lup)$ meets $\mu_1$ exactly when~\eqref{eq:reduce} is an equality. Theorem~\ref{thm:triangle-free} supplies the inequality for every $G$ with triangle-free complement, but not the equality: the self-complementary $G=C_5$ has $\alpha(G)=2$, yet $\Lup=0$ while $\mu_1=\tfrac{5+\sqrt5}{2}$. Equality does occur away from the ceiling, for instance at $K_{n-1}\cup K_1$ with $n\ge4$, a non-join with $\lambda_{\max}(\Lup)=\mu_1=n-1$.
\end{rem}

\subsection{The second largest Helmholtzian eigenvalue}\label{subsec:second}

We now turn to $\lambda_2(L_1)$ and prove Theorem~\ref{thm:lambda2}. The whole argument runs through the spectral decomposition~\eqref{eq:spec-decomp}. We begin by fixing the family that will appear.

\begin{defn}\label{def:firefly}
For integers $s,t,k\ge 0$, the \emph{firefly graph} $F_{s,t,k}$ is obtained from a center vertex $v$ by attaching $s$ triangles, $t$ pendant paths of length two, and $k$ pendant edges, where the $s+t+k$ attached pieces are vertex-disjoint apart from the common vertex $v$ (see Figure~\ref{fig:firefly}). Thus $F_{s,t,k}$ has $n=2s+2t+k+1$ vertices, and we write it as $F_{s,t,n-2s-2t-1}$ when the order is fixed.
\end{defn}

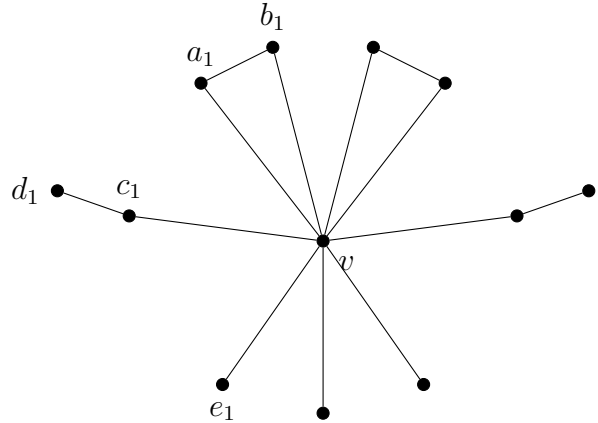
\begin{figure}[ht]
\centering
\begin{tikzpicture}[scale=0.95,
  vtx/.style={circle,draw=black,fill=black,inner sep=1.6pt}]
  \node[vtx,label={below right:$v$}] (v) at (0,0) {};
  \node[vtx,label={above:$a_1$}] (a1) at (-1.7,2.2) {};
  \node[vtx,label={above:$b_1$}] (b1) at (-0.7,2.7) {};
  \node[vtx] (a2) at (0.7,2.7) {};
  \node[vtx] (b2) at (1.7,2.2) {};
  \draw (v)--(a1)--(b1)--(v);
  \draw (v)--(a2)--(b2)--(v);
  \node[vtx,label={above:$c_1$}] (c1) at (-2.7,0.35) {};
  \node[vtx,label={left:$d_1$}] (d1) at (-3.7,0.7) {};
  \draw (v)--(c1)--(d1);
  \node[vtx] (c2) at (2.7,0.35) {};
  \node[vtx] (d2) at (3.7,0.7) {};
  \draw (v)--(c2)--(d2);
  \node[vtx,label={below:$e_1$}] (e1) at (-1.4,-2.0) {};
  \node[vtx] (e2) at (0,-2.4) {};
  \node[vtx] (e3) at (1.4,-2.0) {};
  \draw (v)--(e1);
  \draw (v)--(e2);
  \draw (v)--(e3);
\end{tikzpicture}
\caption{The firefly graph $F_{2,2,3}$. The center $v$ carries $s=2$ triangles, i.e.\ the blades $\{v,a_i,b_i\}$, together with $t=2$ pendant paths of length two and $k=3$ pendant edges. The blades are the only triangles and are pairwise edge-disjoint, so each contributes the up Laplacian eigenvalue $3$.}
\label{fig:firefly}
\end{figure}

The structural input on the Laplacian side is the classification of Li, Guo and Shiu.

\begin{thm}[Li, Guo and Shiu~\cite{LGS}]\label{thm:lgs}
Let $G$ be a connected graph of order $n\ge 7$. Then $\mu_2(G)\le 3$ if and only if $G$ is a firefly graph $F_{s,t,n-2s-2t-1}$. Moreover, for the triangle-free fireflies,
\[
\mu_2\big(F_{0,0,n-1}\big)=1,\qquad
\tfrac{3+\sqrt5}{2}-\tfrac1n<\mu_2\big(F_{0,1,n-3}\big)<\tfrac{3+\sqrt5}{2},\qquad
\mu_2\big(F_{0,t,n-2t-1}\big)=\tfrac{3+\sqrt5}{2}\ \ (t\ge 2).
\]
\end{thm}

The only Helmholtzian-specific fact we need is that the triangles of a firefly contribute exactly the eigenvalue $3$.

\begin{lem}\label{lem:firefly-up}
Let $F=F_{s,t,k}$ with center $v$. The triangles of $F$ are the $s$ blades through $v$, and they are pairwise edge-disjoint. Consequently $\Spec^{+}\!\big(\Lup(F)\big)=\{3^{(s)}\}$ and $b_1(F)=0$, so that
\[
\Spec\big(L_1(F)\big)=\big\{\mu_1(F),\dots,\mu_{n-1}(F)\big\}\ \uplus\ \{3^{(s)}\}.
\]
\end{lem}

\begin{proof}
Every vertex other than $v$ has all of its neighbors inside $\{v\}$ together with its own attached piece, and only the blades carry an edge opposite to $v$; hence a set of three pairwise adjacent vertices must be a blade $\{v,a_i,b_i\}$, and there are exactly $s$ of them. Distinct blades meet only at $v$, so they share no edge and span $3s$ distinct edges. Writing $\partial_2$ for the edge--triangle boundary, each blade boundary $v_S=\partial_2 S$ has $\|v_S\|^2=3$, and edge-disjoint blades give $v_S^{\top}v_{S'}=0$ for $S\ne S'$; thus $\partial_2^{\top}\partial_2=3I_s$. The nonzero eigenvalues of $\Lup=\partial_2\partial_2^{\top}$ are those of $\partial_2^{\top}\partial_2$, namely $3$ with multiplicity $s$. Finally $\rank\partial_2=s$, while $F$ has $m=3s+2t+k$ edges and $n-1=2s+2t+k$, so $b_1=m-(n-1)-\rank\partial_2=0$. The displayed spectrum is now Proposition~\ref{prop:spec-decomp}.
\end{proof}

\begin{proof}[Proof of Theorem~\ref{thm:lambda2}]
Suppose first that $\lambda_2(L_1)\le 3$. By Proposition~\ref{prop:spec-decomp} we have $\lambda_2(L_1)\ge\mu_2(G)$, so $\mu_2(G)\le 3$, and Theorem~\ref{thm:lgs} forces $G$ to be a firefly graph.

Conversely, let $G=F_{s,t,n-2s-2t-1}$ be a firefly of order $n\ge 7$. In a firefly the center is the unique vertex of maximum degree, and the configurations with $n\ge 7$ all have $\deg(v)\ge 3$; hence $\mu_1\ge\deg(v)+1\ge 4>3$ by the bound of Grone, Merris and Sunder~\cite{GMS}, so $\mu_1$ is the largest element of $\Spec(L_1)$. By Lemma~\ref{lem:firefly-up},
\[
\Spec(L_1)=\{\mu_1,\dots,\mu_{n-1}\}\ \uplus\ \{3^{(s)}\},
\]
and removing one copy of $\mu_1$ leaves second largest element $\max\{\mu_2,\,3\cdot[s\ge 1]\}$. If $s=0$, then $F$ is triangle-free, the value $3$ is absent, and $\lambda_2(L_1)=\mu_2(G)$; Theorem~\ref{thm:lgs} then gives the first three cases. If $s\ge 1$, then $\mu_2(G)\le 3$ because $F$ is a firefly, so $\lambda_2(L_1)=\max\{\mu_2,3\}=3$, the fourth case. Since $1<\tfrac{3+\sqrt5}{2}-\tfrac1n<\tfrac{3+\sqrt5}{2}<3$ for $n\ge 7$, the four values are distinct, and the stated equivalences follow.
\end{proof}

\begin{rem}\label{rem:intrinsic-three}
The proof isolates the role of the triangles: the boundary value $3$ in Theorem~\ref{thm:lambda2} is the up Laplacian eigenvalue $\|\partial_2 S\|^2=3$ of a single triangle $S$, surfacing through~\eqref{eq:spec-decomp} once $\mu_2\le 3$ has fixed the Laplacian part. The mechanism is not special to fireflies. For any connected graph $G$ whose triangles are pairwise edge-disjoint one has $\Spec^{+}(\Lup)=\{3^{(t_\triangle(G))}\}$, where $t_\triangle(G)$ is the number of triangles, and therefore
\[
\lambda_2(L_1)=\max\{\mu_2(G),\,3\cdot[\,t_\triangle(G)\ge 1\,]\}.
\]
Indeed, if $G$ has a triangle then $\Delta(G)\ge2$, so $\mu_1\ge\Delta(G)+1\ge3$ and $\mu_1$ is again the largest element of $\Spec(L_1)$; if $G$ has no triangle the displayed value is $\mu_2$, as $\Lup=0$.
Among connected graphs of order $n\ge 7$ the constraint $\mu_2\le 3$ already forces the firefly structure, which is why no graphs beyond fireflies appear; but the displayed identity holds verbatim for the edge-disjoint-triangle regime in general.
\end{rem}

\section{Higher Dimensions}\label{sec:general}

The argument of Section~\ref{sec:main} is not particular to graphs. In this section we carry it out in every dimension and prove the three results announced in the introduction: the general bound (Theorem~\ref{thm:maink-intro}, proved here as Theorem~\ref{thm:maink}), which recovers Theorem~\ref{thm:main} when $k=1$; the unconditional case of Theorem~\ref{thm:tf-k-intro}, proved as Theorem~\ref{thm:tf-k}; and the shifted case of Theorem~\ref{thm:shifted-intro}, proved as Theorem~\ref{thm:shifted}.

Recall that an \emph{abstract simplicial complex} on $[n]$ is a family of subsets of $[n]$ closed under inclusion; a member of size $j+1$ is a \emph{$j$-face}. The clique complex of Sections~\ref{sec:def}--\ref{sec:open} is the instance whose faces are the vertex sets of the cliques of a graph, and its faces of dimension at most two carry the Helmholtzian. We now allow faces of every dimension; the boundary operators, chain spaces, and Laplacians below are those of the general combinatorial Hodge theory of Horak and Jost~\cite{HJ}.

Fix an integer $k\ge1$. We work in the full simplex on $[n]$, writing $C_j=\mathbb{R}^{\binom{[n]}{j+1}}$ for its space of $j$-chains and $\partial_j\colon C_j\to C_{j-1}$ for the boundary maps, so that $L_j^{\mathrm{up}}=\partial_{j+1}\partial_{j+1}^{\top}$ and $L_j^{\mathrm{down}}=\partial_j^{\top}\partial_j$ as in Section~\ref{sec:def}. We use the augmented convention throughout: $C_{-1}=\mathbb{R}^{\binom{[n]}{0}}=\mathbb{R}$ and $\partial_0=\mathbf 1^{\top}$, so that $Z_0=\ker\partial_0=\mathbf 1^{\perp}$ and $L_0^{\mathrm{up}}(G)=L(G)$. This is what makes the case $k=1$ below specialize to Section~\ref{sec:main}: without the augmentation one would have $Z_0=C_0$ and $\lambda_{\min}(L(\Gbar)|_{Z_0})=0$ rather than $a(\Gbar)$. For a simplicial complex $G$ on $[n]$, its up Laplacian in dimension $k$ is $L_k^{\mathrm{up}}(G)=\sum_{F}v_Fv_F^{\top}$, summed over the $(k+1)$-faces $F$ of $G$, where $v_F=\partial_{k+1}F\in C_k$; as in Section~\ref{sec:main} this acts on $C_k$ and annihilates $Z_k^{\perp}$, where $Z_k=\ker\partial_k$. Write
\[
\nu_{k-1}(G)=\lambda_{\max}\big(L_{k-1}^{\mathrm{up}}(G)\big),
\]
so that for $k=1$, $\nu_0(G)=\lambda_{\max}(L(G))=\mu_1(G)$ is the largest Laplacian eigenvalue of the introduction.

Each $v_F$ lies in $Z_k$, and on the full simplex $\big(\sum_{F\in\binom{[n]}{k+2}}v_Fv_F^{\top}\big)\big|_{Z_k}=nI$ (see~\cite{HJ}). Writing $\bar L_k(G)$ for the sum of $v_Fv_F^{\top}$ over the \emph{missing} $(k+1)$-faces of $G$, we obtain on $Z_k$ the identity $L_k^{\mathrm{up}}(G)|_{Z_k}+\bar L_k(G)|_{Z_k}=nI$, exactly as in~\eqref{eq:complement}. The same identity one dimension lower, applied to the $k$-faces, reads $L_{k-1}^{\mathrm{up}}(G)|_{Z_{k-1}}+L_{k-1}^{\mathrm{up}}(\bar G)|_{Z_{k-1}}=nI$, where $L_{k-1}^{\mathrm{up}}(\bar G)=\sum_{\tau}v_{\tau}v_{\tau}^{\top}$ is summed over the missing $k$-faces $\tau$, with $v_{\tau}=\partial_k\tau$. We therefore set
\[
a_k(\bar G):=\lambda_{\min}\big(L_{k-1}^{\mathrm{up}}(\bar G)\big|_{Z_{k-1}}\big)=n-\nu_{k-1}(G),
\]
the codimension-one connectivity of the complement, which generalizes $a(\Gbar)=n-\mu_1(G)$.

For a $k$-face $\tau$ (a $(k+1)$-subset of $[n]$) and a vertex $w\notin\tau$, let $\crc_{\tau,w}(x)=\langle\partial_{k+1}(\tau\cup w),\,x\rangle$. As in Section~\ref{sec:def}, $\crc_{\tau,w}(x)=\varepsilon_w x_{\tau}+D_{\tau,w}(x)$, where $\varepsilon_w\in\{\pm1\}$ is the sign of the facet $\tau$ in $\partial_{k+1}(\tau\cup w)$ and
\[
D_{\tau,w}(x)=\sum_{u\in\tau}\varepsilon_u\,x_{(\tau\setminus u)\cup w}
\]
collects the remaining $k+1$ facets. The following is the higher-dimensional form of Lemma~\ref{lem:identity}.

\begin{lem}\label{lem:identityk}
Let $x\in Z_k$ and let $\tau$ be a $(k+1)$-subset of $[n]$. Then
\[
\sum_{w\notin\tau}\crc_{\tau,w}(x)^2=(n+k+1)\,x_{\tau}^2+\sum_{w\notin\tau}D_{\tau,w}(x)^2 .
\]
\end{lem}

\begin{proof}
Since $\varepsilon_w^2=1$, squaring and summing the $n-k-1$ terms gives
\[
\sum_{w\notin\tau}\crc_{\tau,w}(x)^2=(n-k-1)x_{\tau}^2+2x_{\tau}\sum_{w\notin\tau}\varepsilon_wD_{\tau,w}(x)+\sum_{w\notin\tau}D_{\tau,w}(x)^2 .
\]
For the middle term, write $\varepsilon(v,g)$ for the sign of inserting a vertex $v$ into a face $g$, so that $(\partial_kx)_g=\sum_{v\notin g}\varepsilon(v,g)\,x_{g\cup v}$ for every $k$-subset $g$, and let $i(u)$ denote the position of $u$ in the increasing order of $\tau$. Fix $u\in\tau$ and put $g_u=\tau\setminus u$. Comparing positions inside $\tau\cup w$ and inside $\tau$ gives
\[
\varepsilon(u,g_u)=(-1)^{i(u)},\qquad
\varepsilon_w\,\varepsilon\big(u,(\tau\setminus u)\cup w\big)=(-1)^{i(u)+1}\varepsilon(w,g_u),
\]
the second identity being independent of $w$: the position of $u$ in $\tau\cup w$ exceeds $i(u)$ by $[\,w<u\,]$, and the correction from deleting $u$ contributes $[\,u<w\,]$, so the two shifts always sum to~$1$. Now $g_u$ is a $(k-1)$-face whose completing vertices are $u$ and the vertices $w\notin\tau$, so $\partial_kx=0$ at $g_u$ reads $\varepsilon(u,g_u)x_{\tau}+\sum_{w\notin\tau}\varepsilon(w,g_u)x_{g_u\cup w}=0$. Multiplying by $(-1)^{i(u)+1}$ turns this into $\sum_{w\notin\tau}\varepsilon_w\varepsilon(u,(\tau\setminus u)\cup w)x_{(\tau\setminus u)\cup w}=x_{\tau}$, and summing the $k+1$ resulting identities over $u\in\tau$ gives $\sum_{w\notin\tau}\varepsilon_wD_{\tau,w}(x)=(k+1)x_{\tau}$. The middle term therefore equals $2(k+1)x_{\tau}^2$. Adding $(n-k-1)x_{\tau}^2+2(k+1)x_{\tau}^2=(n+k+1)x_{\tau}^2$ completes the proof.
\end{proof}

\begin{thm}\label{thm:maink}
Let $k\ge1$ and let $G$ be a simplicial complex on $[n]$. Then
\[
\lambda_{\max}\big(L_k^{\mathrm{up}}(G)\big)\ \le\ \nu_{k-1}(G)+\frac{1}{k+2}\big(n-\nu_{k-1}(G)\big).
\]
\end{thm}

\begin{proof}
As in Proposition~\ref{prop:cycle}, $\lambda_{\max}(L_k^{\mathrm{up}}(G))=n-\lambda_{\min}(\bar L_k(G)|_{Z_k})$, so it suffices to prove that $x^{\top}\bar L_k(G)\,x\ge\frac{k+1}{k+2}\,a_k(\bar G)$ for every $x\in Z_k$ with $\|x\|=1$.

\smallskip
\noindent\emph{Step 1 (overcounting).}
Each missing $(k+1)$-face $\sigma$ has $\omega(\sigma)\in\{0,1,\dots,k+2\}$ facets that are missing $k$-faces; for a general complex the value $0$ can occur, when every facet of $\sigma$ is present (a hollow simplex), and such $\sigma$ contribute to $\bar L_k(G)$ but not to $Q$, so they only strengthen the comparison below. Since $\omega(\sigma)\le k+2$,
\[
\bar L_k(G)\ \succeq\ \frac1{k+2}\sum_{\sigma}\omega(\sigma)\,v_{\sigma}v_{\sigma}^{\top}=\frac1{k+2}\,Q,\qquad
Q:=\sum_{\tau}\ \sum_{w\notin\tau}v_{\tau\cup w}v_{\tau\cup w}^{\top},
\]
the inner sum running over the missing $k$-faces $\tau$. By Lemma~\ref{lem:identityk},
\[
x^{\top}Qx=\sum_{\tau}\Big[(n+k+1)x_{\tau}^2+\sum_{w\notin\tau}D_{\tau,w}(x)^2\Big].
\]

\smallskip
\noindent\emph{Step 2 (the complement's connectivity, via signed contraction).}
For $w\in[n]$ define $h^{w}\in C_{k-1}$ by $h^{w}_{g}=\varepsilon(w,g)\,x_{g\cup w}$ for $g\not\ni w$ and $h^{w}_{g}=0$ otherwise, where $\varepsilon(w,g)$ is the sign of inserting $w$ into $g$. Then $\partial_{k-1}h^{w}=0$ follows from $\partial_kx=0$, so $h^{w}\in Z_{k-1}$. Moreover $\langle\partial_k\tau,\,h^{w}\rangle=-\varepsilon_w\,D_{\tau,w}(x)$ when $w\notin\tau$, while for $w\in\tau$ the same pairing equals $x_{\tau}$; only the squares of these quantities enter below. Hence, summing over the missing $k$-faces $\tau$,
\[
\sum_{w\in[n]}(h^{w})^{\top}L_{k-1}^{\mathrm{up}}(\bar G)\,h^{w}=\sum_{\tau}\Big[\sum_{w\notin\tau}D_{\tau,w}(x)^2+(k+1)x_{\tau}^2\Big].
\]
Since each $h^{w}\in Z_{k-1}$, the left side is at least $a_k(\bar G)\sum_w\|h^{w}\|^2$, and $\sum_w\|h^{w}\|^2=(k+1)\|x\|^2=k+1$, because every $(k+1)$-subset arises as $g\cup w$ in $k+1$ ways. Therefore
\[
\sum_{\tau}\sum_{w\notin\tau}D_{\tau,w}(x)^2\ \ge\ (k+1)\,a_k(\bar G)-(k+1)\sum_{\tau}x_{\tau}^2 .
\]

\smallskip
\noindent\emph{Step 3 (combine).}
Substituting into the expression for $x^{\top}Qx$,
\[
x^{\top}Qx\ \ge\ (n+k+1)\sum_{\tau}x_{\tau}^2+(k+1)a_k(\bar G)-(k+1)\sum_{\tau}x_{\tau}^2
=n\sum_{\tau}x_{\tau}^2+(k+1)a_k(\bar G),
\]
which is at least $(k+1)a_k(\bar G)$.
Hence $x^{\top}\bar L_k(G)x\ge\frac1{k+2}x^{\top}Qx\ge\frac{k+1}{k+2}a_k(\bar G)$, and
\[
\lambda_{\max}(L_k^{\mathrm{up}})=n-\lambda_{\min}(\bar L_k|_{Z_k})\le n-\tfrac{k+1}{k+2}a_k(\bar G)=\nu_{k-1}+\tfrac1{k+2}\big(n-\nu_{k-1}\big).\qedhere
\]
\end{proof}

When $k=1$, Theorem~\ref{thm:maink} reduces to Theorem~\ref{thm:main}, with $a_1(\bar G)=n-\mu_1(G)=a(\Gbar)$ the algebraic connectivity of the complement graph. The divergence-free identity and the bound thus extend to every dimension, with $\tfrac13$ replaced by $\tfrac1{k+2}$.

The complement condition of Theorem~\ref{thm:triangle-free} also has a counterpart in every dimension. The proof of Theorem~\ref{thm:maink} shows that $x^{\top}Qx\ge(k+1)\,a_k(\bar G)\|x\|^2$ for $x\in Z_k$, and the only place the loss enters is the comparison $\bar L_k(G)\succeq\frac1{k+2}Q$, which weakens to an equality-friendly $\bar L_k(G)\succeq\frac1{k+1}Q$ once no missing face is full.

\begin{thm}\label{thm:tf-k}
Let $k\ge1$ and let $G$ be a simplicial complex on $[n]$ in which every missing $(k+1)$-face has at most $k+1$ missing facets. Then
\[
\lambda_{\max}\big(L_k^{\mathrm{up}}(G)\big)\ \le\ \nu_{k-1}(G);
\]
that is, Conjecture~\ref{conj:general} holds in dimension $k$. This is Theorem~\ref{thm:tf-k-intro} of the introduction; for $k=1$ it is Theorem~\ref{thm:triangle-free}.
\end{thm}

\begin{proof}
By hypothesis $\omega(\sigma)\le k+1$ for every missing $(k+1)$-face $\sigma$, so
\[
\bar L_k(G)=\sum_{\sigma}v_{\sigma}v_{\sigma}^{\top}\ \succeq\ \frac1{k+1}\sum_{\sigma}\omega(\sigma)\,v_{\sigma}v_{\sigma}^{\top}=\frac1{k+1}\,Q.
\]
The proof of Theorem~\ref{thm:maink} gives $x^{\top}Qx\ge(k+1)\,a_k(\bar G)\|x\|^2$ for $x\in Z_k$, so $x^{\top}\bar L_k(G)x\ge a_k(\bar G)\|x\|^2$, that is $\lambda_{\min}(\bar L_k|_{Z_k})\ge a_k(\bar G)$. Hence $\lambda_{\max}(L_k^{\mathrm{up}})=n-\lambda_{\min}(\bar L_k|_{Z_k})\le n-a_k(\bar G)=\nu_{k-1}$. For $k=1$ the hypothesis $\omega(S)\le 2$ is exactly that $\Gbar$ is triangle-free.
\end{proof}

For shifted complexes---the higher-dimensional analogue of threshold graphs---Conjecture~\ref{conj:general} can be settled in full, and the largest up Laplacian eigenvalue takes an explicit combinatorial value in every dimension. Recall that a simplicial complex $G$ on $[n]$ is \emph{shifted} if, for some ordering of the vertices, replacing any vertex of a face by a smaller one always yields a face; a graph is shifted exactly when it is a threshold graph. For $j\ge1$ and a vertex $v$, let $\deg_j(v)$ be the number of $j$-faces containing $v$, and set
\[
V_j(G)=\big|\{v\in[n]:\deg_j(v)\ge1\}\big|,
\]
the number of vertices lying in some $j$-face. Duval and Reiner determined the up Laplacian spectra of shifted complexes through these degrees: for a shifted complex, the nonzero spectrum of $L_k^{\mathrm{up}}$ is the conjugate partition of the degree sequence $\big(\deg_{k+1}(v)\big)_{v\in[n]}$, by~\cite[Theorem~1.1]{DR} applied to the (shifted) family of $(k+1)$-faces. As the largest part of the conjugate of a partition is its number of nonzero entries, this reads
\begin{equation}\label{eq:shifted-max}
\lambda_{\max}\big(L_k^{\mathrm{up}}(G)\big)=V_{k+1}(G)\qquad(G\text{ shifted},\ k\ge0).
\end{equation}
For $k=0$ this is $\lambda_{\max}(L(G))=V_1(G)$, the number of non-isolated vertices, which equals $n$ for a connected threshold graph and recovers $\mu_1=n$.

\begin{thm}\label{thm:shifted}
Conjecture~\ref{conj:general} holds for every shifted complex: for all $k\ge1$,
\[
\lambda_{\max}\big(L_k^{\mathrm{up}}(G)\big)=V_{k+1}(G)\ \le\ V_k(G)=\lambda_{\max}\big(L_{k-1}^{\mathrm{up}}(G)\big).
\]
Thus on a shifted complex $\lambda_{\max}(L_k^{\mathrm{up}})$ is non-increasing in $k$, with the explicit value $V_{k+1}(G)$ in each dimension.
\end{thm}

\begin{proof}
The two equalities are~\eqref{eq:shifted-max}. For the inequality, suppose $v$ lies in a $(k+1)$-face $\tau$, a set of $k+2\ge2$ vertices. Deleting from $\tau$ any vertex other than $v$ leaves a $k$-face---a subface of $\tau$, hence in $G$---that still contains $v$. Hence $\{v:\deg_{k+1}(v)\ge1\}\subseteq\{v:\deg_k(v)\ge1\}$, that is $V_{k+1}(G)\le V_k(G)$.
\end{proof}

\begin{rem}\label{rem:shifted-eq}
Equality holds in dimension $k$ exactly when every vertex lying in a $k$-face also lies in a $(k+1)$-face; in particular $\lambda_{\max}(L_k^{\mathrm{up}})=\lambda_{\max}(L_{k-1}^{\mathrm{up}})$ throughout the dimensions below the top of a pure shifted complex. The bound of Theorem~\ref{thm:maink} is never binding here: since $V_k(G)\le n$, the exact value already satisfies $V_{k+1}\le V_k\le V_k+\tfrac1{k+2}(n-V_k)$. Shifted complexes are therefore a class on which Conjecture~\ref{conj:general}, open in general, holds in every dimension, with the extremal eigenvalue counted exactly by~\eqref{eq:shifted-max}. The theorem is Theorem~\ref{thm:shifted-intro} of the introduction.
\end{rem}

\section{Concluding remarks}\label{sec:open}

Proposition~\ref{prop:cycle} recasts inequality~\eqref{eq:reduce} as the spectral inequality
\begin{equation}\label{eq:refor}
\lambda_{\min}\!\big(\bar L(G)\big|_{\Zone}\big)\ \ge\ a(\Gbar),
\end{equation}
between the up Laplacian of the missing triangles of $G$, restricted to the cycle space, and the algebraic connectivity of the complement. Theorem~\ref{thm:main} proves it with $a(\Gbar)$ replaced by $\tfrac23\,a(\Gbar)$, and Theorem~\ref{thm:triangle-free} proves~\eqref{eq:refor} itself when $\Gbar$ is triangle-free. The right-hand side $a(\Gbar)=n-\mu_1(G)$ vanishes exactly when $G$ is a join, so~\eqref{eq:refor} carries content only when $\mu_1(G)<n$.

The constant on the right of~\eqref{eq:refor} cannot be improved. For $G=K_{n-1}\cup K_1$ with $n\ge4$ one has $\lambda_{\min}(\bar L|_{\Zone})=a(\Gbar)=1$, so~\eqref{eq:refor} is an equality there and Theorem~\ref{thm:main} recovers exactly two-thirds of it. Which graphs achieve equality is a question that~\eqref{eq:refor} itself does not settle. Both sides of $\lambda_{\max}(\Lup)=\mu_1$ are maxima over the components of $G$, so equality forces some component $H$ to satisfy $\lambda_{\max}(\Lup(H))=\mu_1(H)=\mu_1(G)$, and in every instance we know that component attains the ceiling of Corollary~\ref{cor:ceiling}, that is $\mu_1(H)=|V(H)|$; the graph $K_{n-1}\cup K_1$ is the case $H=K_{n-1}$. We know no connected non-join with $\lambda_{\max}(\Lup)=\mu_1$, and none exists on at most eight vertices. Ruling them out, however, cannot be separated from~\eqref{eq:refor}: a connected $G$ with $\mu_1(G)<n$ and $\lambda_{\max}(\Lup)=\mu_1$ would in particular satisfy $\lambda_{\max}(\Lup)\le\mu_1$, so excluding such graphs is no easier than the inequality itself.

Inequality~\eqref{eq:refor} is the graph case of Conjecture~\ref{conj:general}, and the three results of this paper stand to it as follows: Theorem~\ref{thm:maink} gives the conjecture up to the loss $\tfrac1{k+2}(n-\nu_{k-1})$ in every dimension, Theorem~\ref{thm:tf-k} gives it whenever every missing $(k+1)$-face has at most $k+1$ missing facets, and Theorem~\ref{thm:shifted} gives it unconditionally for shifted complexes, with the extremal eigenvalue counted by~\eqref{eq:shifted-max}. What the two unconditional cases have in common is that the loss enters at a single point, the comparison $\bar L_k(G)\succeq\frac1{k+2}Q$ of Step~1; closing the conjecture in general appears to require a weighting of the missing faces that is not uniform.

The decomposition~\eqref{eq:spec-decomp} behind Theorem~\ref{thm:lambda2} also bears on the inverse problem for $L_1$: a connected graph $L_1$-cospectral with a firefly inherits, through~\eqref{eq:spec-decomp}, the first Betti number $b_1=0$ and the triangle count $t_\triangle(G)=\tfrac13\big(\operatorname{tr}L_1-2m\big)$, so the data needed to recover the firefly are already visible in the spectrum. We expect that every firefly is determined by its Helmholtzian spectrum among connected graphs.

\end{document}